\documentclass[12pt, reqno]{amsart}
\usepackage{amsmath, amsthm, amscd, amsfonts, amssymb, graphicx, color,verbatim}
\usepackage[bookmarksnumbered, colorlinks, plainpages]{hyperref}
\hypersetup{colorlinks=true,linkcolor=red, anchorcolor=green, citecolor=cyan, urlcolor=red,
filecolor=magenta, pdftoolbar=true}

\newtheorem{theorem}{Theorem}[section]
\newtheorem{lemma}[theorem]{Lemma}
\newtheorem {question}[theorem]{Question}
\newtheorem{proposition}[theorem]{Proposition}
\newtheorem{corollary}[theorem]{Corollary}
\theoremstyle{definition}

\newtheorem{example}[theorem]{Example}

\newtheorem{remark}[theorem]{Remark}
\numberwithin{equation}{section}

\begin{document}

\setcounter{page}{1}

\title[Shift invariant subspaces of composition operators]{Composition operators between Beurling subspaces of Hardy space}
\author[V. A. Anjali, P. Muthukumar \MakeLowercase{and} P. Shankar]{V. A. Anjali, P. Muthukumar \MakeLowercase{and} P. Shankar}


\address{V. A. Anjali, Department of Mathematics, Cochin University of Science And Technology,
 Ernakulam, Kerala- 682022, India. }
\email{\textcolor[rgb]{0.00,0.00,0.84}{anjaliva6446@gmail.com}}

\address{P. Muthukumar, Department of Mathematics and Statistics, Indian Institute of Technology,
 Kanpur- 208016, India. }
\email{\textcolor[rgb]{0.00,0.00,0.84}{pmuthumaths@gmail.com, muthu@iitk.ac.in}}

\address{P. Shankar, Department of Mathematics, Cochin University of Science And Technology,
 Ernakulam, Kerala- 682022, India.}
\email{\textcolor[rgb]{0.00,0.00,0.84}{shankarsupy@gmail.com, shankarsupy@cusat.ac.in}}



\date{\today
}

\begin{abstract}
V. Matache (J. Operator Theory 73(1):243--264, 2015) raised an open problem about characterizing composition 
operators $C_{\phi}$ on the Hardy space $H^2$ and nonzero singular measures $\mu_1$, $\mu_2$ on the unit
 circle such that
	$C_{\phi}({S_{\mu_1}} H^2)\subseteq {S_{\mu_2}} H^2,$ where $S_{\mu_i}$ denotes the singular inner function
 corresponding to the measure $\mu_i,i=1,2$. In this article, we consider this problem in a more general setting.
We characterize holomorphic self maps $\phi$  of the unit disk $\mathbb{D}$ and inner functions
 $\theta_1, \theta_2$ such that
$C_{\phi}(\theta_1  H^p)\subseteq \theta_2 H^p,$ for $p>0$. Emphasis is given to Blaschke products
 and singular inner functions as a special case. 
 We also give an another measure-theoretic
characterization to above question  when $\phi$ is an elliptic automorphism.
 For a given Blaschke product $\theta$, we discuss about finding all
 self maps $\phi$ such that $\theta H^p$ is invariant under $C_\phi$.

\end{abstract}
\maketitle

\section{Introduction}
 Let $\mathbb{D}$ be the open unit disk in the complex plane and $\phi$ be any holomorphic self
 map on $\mathbb{D}$.
 For any linear space $\mathcal{V}$ of holomorphic functions on $\mathbb{D}$, the composition
  operator $C_{\phi}$, is defined as
$$C_{\phi}(f)=f \circ \phi,\text{ for all } f\in \mathcal{V}.$$ The study of composition operators
 has vastly done
in the classical analytic function spaces like Hardy spaces, Bergman spaces, Dirichlet spaces and
so on, in the context
 of boundedness, compactness and
various other operator theoretic properties. See \cite{cowenbook} and references therein for more details.
It is worth noting that $C_\phi$ maps every Hardy space $H^p$ into itself for any holomorphic self map $\phi$.

By an invariant subspace of an operator $T$, we mean a closed linear subspace which is invariant under $T$.
  Beurling \cite{beurling} identified all the invariant subspaces of multiplication
 operator $M_z$ (also, commonly known as shift operator) on $H^2$ induced by the coordinate function.
 He proved that $\{\theta H^2: \theta \mbox{~is inner}\}$ is the set of all
nontrivial  invariant subspace of $M_z$ on $H^2$.
On account of the above result, for $p>0$ and an inner function $\theta$, we call $\theta H^p$  as
\textit{Beurling subspace}.

The renowned ``Invariant Subspace Problem (ISP)", apparently arose after Beurling's work, which addresses
 the question
of whether every bounded linear operator on an infinite dimensional separable Hilbert space possesses
 a nontrivial invariant subspace. In \cite{COMP}, Nordgren et al., gave an equivalent problem for
  ISP in terms of
  composition operators induced by the hyperbolic automorphism on $H^2$. Recently, Carmo and Noor \cite{noor}
  reformulated the ISP in terms of composition operators induced by hyperbolic maps on $\mathbb{D}$.
   This increased
   the interest in understanding the invariant subspaces of composition operators on  $H^p$ spaces.
     In particular,
    the study of Beurling subspaces invariant under composition operators has become an important and
    interesting
    topic of research in operator theory.
	 	
 Mahvidi \cite{mahvidi} considered the common invariant subspaces of two composition operators and the
  lattice containment
  for two composition operators. All invariant subspaces of composition operators on $H^2$
  induced by parabolic non-automorphism were determined in \cite{parabolic}.
   Chalender and Partington \cite{Chalendar} initiated the study of
   Beurling subspaces which are
  invariant under composition operators. Jones \cite{jones} investigated invariant Beurling subspaces  of
  composition operator $C_\phi$ when $\phi$ is an inner function.  Cowen and Wahl \cite{cowen}
   proved that if $\phi$ has the Denjoy-Wolff point $a$ on the unit circle,
   then the atomic inner function
subspaces with a single atom at $a$ are invariant subspaces for the composition operator $C_\phi$.
 Matache \cite{valentine} proved that every composition operator on $H^2$ has
   a nontrivial invariant Beurling  subspace. Bose, Muthukumar and Sarkar \cite{buerlingtype} unified
    the observations from
    \cite{cowen, jones, valentine} and identified a characterization for Beurling subspace $\theta H^2$ to
    be invariant
     under $C_\phi$ in terms of $\theta $ and $\phi$. In a subsequent paper, Muthukumar and
     Sarkar \cite{model}
     explored  model spaces that are invariant under composition operator $C_\phi$ on $H^2$.

In \cite[Problem 1]{valentine}, Matache raised the following question. In the same paper, this question is
 answered when both the measures  $\mu_1$ and  $\mu_2$  are purely atomic.
\begin{question}\label{qn}
	For any holomorphic self map $\phi$ of $\mathbb{D}$ and nonzero singular measures $\mu_1$, $\mu_2$ on
 the unit circle,
what characterization can be given for
	$$C_{\phi}({S_{\mu_1}} H^2)\subseteq{S_{\mu_2}} H^2? $$
\end{question}

In this article, we consider the above problem in a more general setting. In Section 3,
we characterize holomorphic self maps $\phi$  of the unit disk $\mathbb{D}$ and inner functions
 $\theta_1, \theta_2$ such that
$C_{\phi}(\theta_1  H^p)\subseteq \theta_2 H^p,$ for $p>0$. Along with
several interesting consequences of this characterization, we also determine all the
inner functions $\theta$ such that $C_{\phi}(\theta  H^p)\subseteq BH^p$  for a given arbitrary
 Blaschke product $B$.

In Section 4, we restrict our attention to Beurling subspaces induced by
singular inner functions. In this case, we also give another measure-theoretic
characterization as an answer to the Question \ref{qn}  when $\phi$ is an elliptic automorphism.

In Section 5, for a given inner function $\theta$, we attempt to collect
all maps $\phi$ such that $\theta H^p$ invariant under $C_\phi$.  In particular, we focus on the case
when the inner function $\theta$ is a Blaschke product $B$. Derivatives of $\phi$ at the zeros of $B$
plays a very crucial role to know whether $C_{\phi}(BH^p)\subseteq BH^p.$ Various special cases are
considered to understand
the results more deeply.
 Also, we find a class of Beurling subspaces which are not invariant under any composition operator
  induced by a nontrivial automorphism.

\section{Preliminaries}
In this section, we present some notations and the necessary background for what follows.
 Let $\mathbb{N}$ denote
the set of all natural numbers. We denote the open unit disk and unit circle in the complex
plane as $\mathbb{D}$ and $\mathbb{T}$,
 respectively.
For $0<p<\infty$, the Hardy space $H^p(\mathbb{D})$ or simply $H^p$ is defined as the set of
 all holomorphic functions $f$ on $\mathbb{D}$ such that 
$$
 \Vert f\Vert_{p}=\sup\limits_{0\leq r<1}\Big(\dfrac{1}{2\pi}\int\limits_{0}^{2\pi}
|f(re^{i\theta})|^pd\theta\Big)^{\frac{1}{p}}
$$ 
is finite.
$H^{\infty}$ denotes the algebra of all bounded holomorphic functions $f$ on $\mathbb{D}$
with supremum norm.
The closed unit ball of $H^{\infty}$  is denoted as $\mathcal{B}_1$.  The collection of all holomorphic
 self maps on $\mathbb{D}$ is denoted by
$\mathcal{S}$. As a consequence of the Schwarz lemma, it is well known that if $\phi\in \mathcal{S}$ has
more than two fixed points in $\mathbb{D}$ then $\phi$ must be the identity map. This fact will be used
many times in the article. The reader can refer \cite{duren,texthardy} for an introduction to the theory
 of Hardy spaces. It is trivial to see that if $\psi\in H^{\infty}$, then
$\psi f\in H^p$ for all $f\in H^p$.

By Fatou's theorem {\cite[Theorem 2.2]{duren}} for $f\in H^p$ $(0<p\leq\infty)$, the \textit{radial limit}
$$
\tilde{f}(e^{it})=\lim_{r \rightarrow1^-}f(re^{it})
$$ 
exist almost everywhere (a.e.) on  $\mathbb{T}$
 and $\tilde{f}\in L^p(\mathbb{T})$ (with respect to Lebesgue measure on $\mathbb{T}$) with
  $\Vert f\Vert_{p}=\Vert \tilde{f}\Vert_{L^p(\mathbb{T})}$.  It is trivial to see that for any
   $\phi\in\mathcal{S}$, we have $\vert\tilde{\phi}{(e^{it})}\vert\leq1$ a.e. on $\mathbb{T}.$
   As a partial converse of Fatou's theorem, we have the following result.
\begin{theorem}\cite[Theorem 2.11]{duren}\label{bdd}
	Let $f\in H^p$ for some $p>0$. If $\tilde{f}\in L^{\infty}(\mathbb{T}) $, then $f\in H^{\infty}$
 with $\Vert f\Vert_{H^\infty}=\Vert \tilde{f}\Vert_{L^{\infty}(\mathbb{T})}$.
\end{theorem}

A function $\theta\in H^{\infty}$ is said to be an \textit{inner function} if  $|\tilde{\theta}(e^{it})|=1$
a.e. on $\mathbb{T}$. Any inner function $\theta$ can be factorized as
$\theta=BS,$
where $B$ is a Blaschke product and $S$ is a singular inner function, which we referred as the inner
 factorization of the function $\theta$ \cite[Corollary 2.6.6]{texthardy}. It is also important to
 note that this factorization is
 unique up to unimodular constants.

Every Blaschke product $B$ will be of the form
$$B(z)=\gamma z^m\prod_{i=1}^{\infty}\dfrac{|a_i|}{a_i}\dfrac{a_i-z}{1-\overline{a_i}z}
$$
 where $m\in \mathbb{N}\cup \{ 0 \}$ and $\{a_i\}$ is a complex sequence (possibly finite) in $\mathbb{D}$
 such that $\sum_{i\in\mathbb{N}}(1-|a_i|)<\infty$ and $\gamma\in \mathbb{T}$. For $a\in \mathbb{D}$,
 we denote
$B_a(z)=(a-z)/(1-\overline{a}z)$,  $z\in \mathbb{D}$. For $i\in \mathbb{N}$, we denote
 $\alpha_i=|a_i|/a_i$, if $a_i\neq0$ and $\alpha_i=-1$, if $a_i=0$. Hence by grouping  zeros,
  any arbitrary Blaschke product can be written as $B=\gamma \prod_{i\in\mathbb{N}}(\alpha_iB_{a_i})^{m_i}$,
   where $a_i$'s are distinct zeros of $B$ with corresponding multiplicities $m_i$'s and
    $\gamma\in \mathbb{T}$. Throughout this article, we use the latter format of $B$ for an arbitrary
     Blaschke product. Also, any non-vanishing inner function (singular inner function) $S$ will be
      of the form
 $$S(z)=S_{\mu}(z)=\alpha \exp\left(-\int\limits_{\mathbb{T}}\dfrac{t+z}{t-z}d\mu(t)\right)~~~~
 (z\in \mathbb{D})$$
  for some finite positive Borel measure $\mu$ on $\mathbb{T}$, which is singular with respect
  to Lebesgue measure on $\mathbb{T}$
  and $\alpha\in\mathbb{T}$.

For $p>0$, any function in $H^p$ has a canonical factorization 
\cite[Theorem 2.5]{duren}.
It states that all zeros of a function in $H^p$ can be factored out. 

\begin{theorem}\label{RFT}[Riesz factorization theorem]
	Let $f\in H^p$ for some $p>0$ and $f\not\equiv 0$. Then, there exists a Blaschke product
$B$ and a non-vanishing function $g$ in $H^p$ such that
	$f =Bg$ with $\Vert f\Vert_{p}=\Vert g\Vert_{p}$.
\end{theorem}


For a function $f\in H^p$, $Z(f)$ denotes the set of all zeros of $f$  inside $\mathbb{D}$ and
 multiplicity of any $w\in Z(f)$ is denoted as ${mult}_f(w)$. 

\section{Composition operators between Beurling subspaces}
Consider the inner functions $\theta_1(z)=\exp({\frac{z+a}{z-a}})$ and $\theta_2(z)=\exp({\frac{z+b}{z-b}})$
 for some $a,b\in \mathbb{T}$.
 Matache \cite{valentine} gave a 
characterization  for 
 $C_{\phi}(\theta_1H^2)\subseteq \theta_2H^2$ in terms
 of the angular derivative and raised the Question \ref{qn} for singular inner function case.
  We  solve this question
 for general inner functions.

Let  $\theta_1$ and  $\theta_2$ be two arbitrary inner functions and suppose
$C_{\phi}(\theta_1H^p)\subseteq \theta_2H^p$.
Then for any $f\in H^p$, there exists $g\in H^p$ such that $(\theta_1f)\circ \phi=g\theta_2.$
In particular by taking $f\equiv 1$, we get $ \theta_1 \circ \phi=g\theta_2$ for some $g\in H^p$.
Therefore, $Z(\theta_2)\subseteq Z(\theta _1\circ \phi),$
which is equivalent to saying that $\phi$ maps $Z(\theta_2)$ into $Z(\theta_1).$
Moreover, we also have  $mult_{\theta_1\circ \phi}(w)\geq mult_{\theta_2}(w)$ for all $w\in Z(\theta_2)$.
\begin{theorem}{\label{main}}
	Let $\theta_1$ and $\theta_2$ be inner functions and $\phi$ be a holomorphic self map on $\mathbb{D}$.
 Then $C_{\phi}(\theta_1H^p)\subseteq\theta_2H^p$ for some $p>0$ if and only if $(\theta_1 \circ \phi)/\theta_2\in H^{\infty}$.
 
\end{theorem}
\begin{proof}
	 Suppose $C_{\phi}(\theta_1H^p)\subseteq\theta_2H^p$. Then $\theta_1 \circ \phi=\theta_2f$
 for some $f\in H^p$.
	That is, $(\theta_1 \circ \phi)/\theta_2=f\in H^p$.
 As
	$\theta_1$ is an inner function and $\phi$ is a self map of $\mathbb{D}$, it is evident that
$|(\theta_1\circ\phi)(e^{it})|\leq1$
 a.e.  on $ \mathbb{T}.$
	Since $\theta_2 $ is an inner function, we get
	$|\tilde{f}(e^{it})|\leq1$ a.e. on $ \mathbb{T}.$ 	Thus by Theorem \ref{bdd},
 $f\in H^{\infty}$ with $\Vert f\Vert_{H^\infty}=\Vert \tilde{f}\Vert_{L^{\infty}}\leq 1$.
	
	For the converse part, suppose that $(\theta_1 \circ \phi)/\theta_2=f\in H^{\infty}$.
	For $h\in H^p$ we get,
	$$C_{\phi}(\theta_1h)=(\theta_1h)\circ\phi=(\theta_1\circ\phi)(h\circ\phi)=
\theta_2f(h\circ\phi)\in \theta_2H^p.$$ Hence the desired result $C_{\phi}(\theta_1H^p)\subseteq\theta_2H^p$ holds.
\end{proof}

\begin{corollary}
	Let  $\theta_1$ and $\theta_2$ be inner functions and $\phi$ be a holomorphic self map on $\mathbb{D}$. Then
	$C_{\phi}(\theta_1H^p)\subseteq\theta_2H^p$ for some $p>0$	if and only if
	$C_{\phi}(\theta_1H^2)\subseteq\theta_2H^2$.
\end{corollary}

\begin{corollary}\label{inner}
	Let $\phi$, $\theta_1$ and $\theta_2$ be inner functions. Then
	$C_{\phi}(\theta_1H^p)\subseteq\theta_2H^p$	if and only if
	$(\theta_1 \circ \phi)/\theta_2$ is an inner function.
\end{corollary}
\begin{proof}
	Suppose $C_{\phi}(\theta_1H^p)\subseteq\theta_2H^p$. Take $f =(\theta_1\circ\phi)/\theta_2$.
 As $\theta_1$,$\theta_2$ and $\phi$ all are inner functions, $ |f(e^{it})|=1$  a.e. on  $\mathbb{T}.$
 By Theorem \ref{main}, $f\in H^{\infty}$ and
	therefore $f$ is an inner function. The converse part follows trivially.
\end{proof}
For $f,g\in H^{\infty}$, we say that ``$f$ divides $g$'' if there exists some $h\in H^{\infty}$
 such that  $g=fh$.
When $\theta_1, \theta_2$ are Blaschke products in Theorem \ref{main}, we have another characterization in terms of the multiplicities of zeros.
\begin{proposition} \label{charb}	
Let $B_1$, $B_2$ be two arbitrary Blaschke products and let $\phi\in \mathcal{S}$.
Then  $C_{\phi}(B_1H^p)\subseteq B_2H^p$ if and only if $ mult_{B_2}(w)\leq{mult}_{B_1\circ \phi}(w)$
 for all $w$ in $Z(B_2)$.
\end{proposition}

\begin{proof}
	Suppose $C_{\phi}(B_1H^p)\subseteq B_2H^p$. Then we have $B_1 \circ \phi={B_2}f$ for
some $f\in H^p$. If $w\in Z(B_2)$ with ${mult}_{B_2}(w)=m$, then $(z-w)^m$ divides $B_2$ and
hence $(z-w)^m$ divides $B_1\circ \phi$ with $mult_{B_1\circ \phi}(w)\geq m$.

	Conversely, suppose that  $ {mult}_{B_2}(w)\leq {mult}_{B_1\circ \phi}(w)$ for all $w\in Z(B_2)$.
 Since $B_1\circ\phi\in H^{\infty}$, as a consequence of 
  Theorem \ref{RFT}, there
 exists a Blaschke product $B_3$ and a non-vanishing function $g\in H^{\infty}$ such that
 $B_1\circ \phi=gB_2B_3$ and thus  $(B_1\circ\phi)/B_2=gB_3\in H^{\infty}$.  Hence by  Theorem \ref{main},
  $C_{\phi}$ maps $B_1H^p$ into $B_2H^p$.
\end{proof}

	For an inner function $\theta$ and $\phi\in \mathcal{S}$, we denote the set of all
inner functions $\theta_1$ such that $C_{\phi}(\theta_1 H^p)\subseteq \theta H^p$ by the notation
 $\mathcal{L}_{\theta,\phi}$.

\begin{lemma}\label{pro}
	If $\theta_1\in\mathcal{L}_{\theta,\phi}$ for some inner function $\theta$ and
$\phi\in \mathcal{S}$, then $\theta_2 \theta_1\in\mathcal{L}_{\theta,\phi}$ for every inner function $\theta_2$.
\end{lemma}
\begin{proof}
	For any inner function  $\theta_2$ and for any $f\in H^p$, we have $$C_\phi(\theta_2 \theta_1f)=(\theta_2\circ\phi)((\theta_1f)\circ\phi)=(\theta_2\circ\phi)\theta g\in \theta H^p,$$
	for some $g\in H^p$. That is, $C_{\phi}(\theta_2\theta_1H^p)\subseteq \theta H^p.$ The desired result follows.
\end{proof}
\begin{proposition}\label{compauto}
	Let $B$ be a Blaschke product and $\phi$ be a disk automorphism. Then $B\circ \phi$ is a Blaschke product.
\end{proposition}
\begin{proof}
	Set $\theta=B\circ\phi$. Then $\theta $ is an inner function. Let $\theta =B_1S_1$ be the inner
factorization of $\theta $. As $B_1\circ\phi^{-1}$ is also an inner function,  $B_1\circ\phi^{-1}=B_2S_2$
 be its inner factorization.
Call the singular inner function $S_1\circ\phi^{-1}$ by $S_3$.
Then, we  have $$B=\theta \circ \phi^{-1}=(B_1S_1)\circ\phi^{-1}=
(B_1\circ\phi^{-1})(S_1\circ\phi^{-1})=B_2S_2S_3.$$
 Since inner factorization is unique up to a unimodular
 constant multiplication, we get $B=\gamma_1B_2$ and $S_2S_3=\gamma_2$ for some  $\gamma_1, \gamma_2\in \mathbb{T}$.
  The latter is possible only when $S_2$ and $S_3$ are unimodular constants and thus $S_1$ is an unimodular constant, say
  $\gamma$.  It yields that
  $\theta =\gamma B_1$, which is a Blaschke product.
\end{proof}
\begin{theorem}\label{split}
	Let $\phi$ be an automorphism on $\mathbb{D}$ and consider the inner functions
$\theta_1=B_1S_1$ and $ \theta_2=B_2S_2$ where $B_1, B_2$ are Blaschke products and $S_1, S_2$
are singular inner functions. Then we have, $C_{\phi}(\theta_1H^p)\subseteq \theta_2 H^p$ if and only if
  $C_{\phi}(B_1H^p)\subseteq B_2 H^p$ and $C_{\phi}(S_1H^p)\subseteq S_2 H^p$.
\end{theorem}
\begin{proof}
	Assume that $C_{\phi}(\theta_1H^p)\subseteq \theta_2 H^p$. By Corollary \ref{inner},
$f=(\theta_1\circ\phi)/\theta_2$ is an inner function and let $f=B_3S_3$ be the inner
 factorization of $f$. Thus, we have
	$$B_2S_2B_3S_3=\theta_2 f =\theta_1\circ\phi=(B_1S_1)\circ\phi=(B_1\circ\phi)(S_1\circ\phi).$$
	By Proposition \ref{compauto}, $B_1\circ\phi$ is a Blaschke product.
Therefore,
$B_1\circ\phi=\gamma_1 B_2B_3$ and $S_1\circ\phi=\gamma_2 S_2S_3$ for some $\gamma_1, \gamma_2\in \mathbb{T}$.
 Hence by Theorem \ref{main}, $C_{\phi}(B_1H^p)\subseteq B_2 H^p$ and $C_{\phi}(S_1H^p)\subseteq S_2 H^p$.
  The converse part follows trivially.
\end{proof}
\begin{remark}\label{zero}
 Theorem \ref{split} holds under a weaker hypothesis, namely, $\phi\in \mathcal{S}$ such
 that $B_1\circ\phi$ is a Blaschke product. Also, it is trivial to see that
 $C_{\phi}(\theta_1H^p)\subseteq \theta_2 H^p$ implies
  $C_{\phi}(B_1H^p)\subseteq B_2 H^p$ for any holomorphic self map $\phi$ of $\mathbb{D}$.
\end{remark}
The following example shows that the  Theorem \ref{split} may fail
for a general
 $\phi\in \mathcal{S}$.
\begin{example}
	Consider $S(z)=\exp(\frac{z+1}{z-1})$. By \cite[Theorem 6.4]{garnett}, there exist $a\in \mathbb{D}$
 and a Blaschke product $B$ such that $B_a\circ B=S$. Since $S$ is not an automorphism, so is not $B$.
  Now consider $\phi=B$, $\theta_1=B_a$ and $\theta_2=S$. Then
	$$\dfrac{\theta_1\circ\phi}{\theta_2}=\dfrac{B_a\circ B}{S}=1\in H^\infty.$$
	However $(S_1\circ\phi)/S_2=1/S\notin H^\infty$, where $S_1$ and $S_2$ are the
 singular components of $\theta_1$ and $\theta_2$ respectively. Thus, Theorem \ref{split} can be
 false even if $\phi$ is an inner function.
\end{example}
 By taking  $\theta _1=\theta_2$ in Theorem \ref{split}, we get the following result.
\begin{corollary}\label{split2}
	Let $\phi$ be an automorphism on $\mathbb{D}$ and let $ \theta=BS$ be its inner
factorization.  Then
 $C_{\phi}(\theta H^p)\subseteq \theta H^p$ if and only if $C_{\phi}(B H^p)\subseteq B H^p$
 and $C_{\phi}(S H^p)\subseteq S H^p$.
\end{corollary}

 The following example shows that
 the above corollary may fail, even if $\phi$ is an inner function.

\begin{example}\label{examp}
	Let $\theta=BS$, where $B(z)=-z$ and $S(z)=\exp(\frac{z+1}{z-1})$. Take $\phi=-\theta$.
 Then
	$$\dfrac{\theta\circ \phi}{\theta}=\dfrac{(BS) \circ \phi}{\theta}=
\dfrac{(B\circ \phi)(S\circ\phi)}{\theta}=S\circ\phi\in H^\infty.$$
	Thus, by Theorem \ref{main}, $C_{\phi}(\theta{H}^p)\subseteq \theta{H}^p$.
	But $$\tilde{\phi}(1)=\lim_{r\rightarrow1^-}\phi(r)=
\lim_{r\rightarrow1^-}r\exp\left(\frac{r+1}{r-1}\right)\neq1.$$ By \cite[Theorem 7]{cowen},
  $S{H}^p$ is not invariant under $C_{\phi}$.
\end{example}
\begin{theorem}\label{theorem.6}
	Let $\phi$ be a holomorphic self map on $\mathbb{D}$ and $B$ be any Blaschke product. Then $\theta$
is an inner function such that $C_{\phi}(\theta  H^p)\subseteq BH^p$ if and only if
$C_{\phi}(B_1H^p)\subseteq BH^p$, where $B_1$ is the Blaschke component in the inner
factorization of $\theta$.
\end{theorem}
\begin{proof}
	Suppose $\theta$ is an inner function such that $C_{\phi}(\theta H^p)\subseteq BH^p$.
Let $\theta=B_1S$ 
be its inner factorization. Then
by Theorem  \ref{main},
	$$\dfrac{(B_1\circ\phi)(S\circ \phi)}{B}=\dfrac{\theta \circ \phi}{B}\in H^{\infty}.$$
	Fix $w\in Z(B)$, let $ {mult}_B(w)=m$ so that $(z-w)^m$ divides $B$. Thus,
$(z-w)^m$ divides $(B_1\circ\phi)(S\circ \phi)$. Since $S\circ \phi$ is non-vanishing,
we have  $(z-w)^m$ divides $(B_1\circ\phi)$. By   Proposition \ref{charb}, we have
$C_{\phi}(B_1H^p)\subseteq BH^p$.
	The converse is true by Lemma \ref{pro}, that is if $C_{\phi}(B_1H^p)\subseteq BH^p$,
then $C_{\phi}(SB_1H^p)\subseteq BH^p$ for any singular inner function $S$.
\end{proof}
Using Proposition \ref{charb} and Theorem \ref{theorem.6}, we arrive at one of the main
results of this section.

\begin{theorem}
	For any self holomorphic function $\phi$ on $\mathbb{D}$ and any arbitrary Blaschke product $B$,
	$$\mathcal{L}_{B,\phi}=\{B_1S:{mult}_B(w) \leq {mult}_{B_1\circ\phi}(w)
\text{  for all } w\in Z(B) \text{ and  } S\text{ is singular}\}.$$
\end{theorem}

\section{Special case: Singular Beurling subspaces}

In this section, we discuss the action of the composition operator between two Beurling
subspaces induced by singular inner functions. As an answer to the Question \ref{qn}, for
 two arbitrary singular inner functions $S_1$ and $S_2$, we find another characterization
 (measure-theoretic in nature) under  which $C_\phi(S_1H^p)\subseteq S_2H^p$, when $\phi $
 is an elliptic automorphism. Recall that
 a disk automorphism, other than identity, with a fixed point inside $\mathbb{D}$ is said to be an
 \textit{elliptic automorphism}.
Before we move to our discussion, let us recall some known results.
	\begin{theorem}\cite[Theorem 2.6.7]{texthardy}\label{innerf}
		Let $S_{\mu_1}$ and $S_{\mu_2}$ be two singular inner functions. Then,
$S_{\mu_1}H^p\subseteq S_{\mu_2}H^p$ if and only if
			 $\mu_2(E)\leq \mu_1(E)$ for every Borel subset $E$ of $\mathbb{T}$.
		
	\end{theorem}
	
	\begin{theorem}\cite[Lemma 3.1]{matacheeigen}\label{matache}
		Let $ \phi$ be an automorphism on ${\mathbb{D}}$ and let $S_\mu$ be a
singular inner function. Then there exist a singular measure $\nu$ such
 that $S_{\mu}\circ \phi$ and $S_{\nu}$ divides each other, where $\nu$ is given by
		\begin{equation}\label{singular}
			\nu(E)=\int\limits_{\phi(E)}\dfrac{1-|\phi(0)|^2}{|t-\phi(0)|^2}d\mu(t) \qquad\qquad
		\end{equation}
		for each Borel subset $E$ of $\mathbb{T}$. In particular, $(S_{\mu}\circ \phi)H^p=S_{\nu}H^p$.
	\end{theorem}
	
	\begin{theorem}\label{text}
		Let $\phi $ be an elliptic automorphism on ${\mathbb{D}}$ with $0$  as the unique fixed
 point in ${\mathbb{D}}$. Suppose $S_{\mu_1}$ and $S_{\mu_2}$ be two singular inner functions.
 Then $C_{\phi}(S_{\mu_1}H^p)\subseteq S_{\mu_2}H^p$
		if and only if $\mu_2(E)\leq\mu_1(\phi(E))$ for every Borel subset $E$ of $\mathbb{T}$.
	\end{theorem}
	\begin{proof}
		Suppose	$C_{\phi}(S_{\mu_1}H^p)\subseteq S_{\mu_2}H^p$. That is,
 $(S_{\mu_1}\circ \phi)H^p\subseteq S_{\mu_2}H^p$.
		By Theorem \ref{matache},  there exist a singular measure $\nu$ such that
$(S_{\mu_1}\circ \phi)H^p=S_{\nu}H^p$. Therefore $S_\nu H^p\subseteq S_{\mu_2}H^p$.
 By Theorem \ref{innerf}, for each Borel subset $E$ of $\mathbb{T}$, $\mu_2(E)\le \nu(E)$.
  Since $\phi(0)=0$, from (\ref{singular}) we have $\nu(E)=\mu_1(\phi(E))$. Hence for
  each Borel subset $E$ of $\mathbb{T}$, we have $\mu_2(E)\leq \mu_1(\phi(E))$.
		
		Conversely, suppose that $\mu_1$ and $\mu_2$ be any two singular measures such that
$\mu_2(E)\leq\mu_1(\phi(E))$ for every Borel subset $E$ of $\mathbb{T}$.
Let $f=(S_{\mu_1}\circ\phi)/S_{\mu_2}.$ Since $\phi$ is an elliptic automorphism
 with  $0$ as a fixed point, we have $\phi(z)=\lambda z$, for some $\lambda\in \mathbb{T}$.
 It gives that $\frac{\phi(t)+\phi(z)}{\phi(t)-\phi(z)}=\frac{t+z}{t-z}$ for all $z$ and $t$.
 Thus for any $z\in \mathbb{D}$,
		\begin{equation*}
			\begin{split}
|f(z)|&=\left|\dfrac{\exp(-\int\limits_{\mathbb{T}}\frac{t+\phi(z)}{t-\phi(z)}d\mu_1(t))}
{\exp(-\int\limits_{\mathbb{T}}\frac{t+z}{t-z}d\mu_2(t))}\right|
=\left|\dfrac{\exp(-\int\limits_{\mathbb{T}}\frac{\phi(t)+\phi(z)}{\phi(t)-\phi(z)}d\mu_1(\phi(t)))}
{\exp(-\int\limits_{\mathbb{T}}\frac{t+z}{t-z}d\mu_2(t))}\right|\\
&=\left|\exp\left(-\int\limits_{\mathbb{T}}\frac{t+z}{t-z}d(\mu_1(\phi)-\mu_2)(t)\right)\right|\\
&=\exp\left(-\int\limits_{\mathbb{T}}Re\frac{t+z}{t-z}d(\mu_1(\phi)-\mu_2)(t)\right)\\
				&=\exp\left(-\int\limits_{\mathbb{T}}\frac{1-|z|^2}{|t-z|^2}d(\mu_1(\phi)-\mu_2)(t)\right)\leq 1.
			\end{split}
		\end{equation*}
Note that as both the integrand and measure are nonnegative, the integral in the last line is nonnegative.
Since $f\in H^\infty$,
		by Theorem  \ref{main} we get that $C_{\phi}$ maps $S_{\mu_1}H^p$ into 
 $S_{\mu_2}H^p$.
	\end{proof}
	\begin{theorem}\label{conj}
		Let $\phi$ be an automorphism with the unique fixed point $\omega\in \mathbb{D}$ and let
$\psi=B_\omega \circ \phi \circ B_\omega$. Also let  $\nu_1$ and $\nu_2$ be the corresponding measures
of $S_{\mu_1}\circ B_\omega$ and $S_{\mu_2}\circ B_\omega$ respectively, as mentioned in  Theorem
\ref{matache}. Then the following are equivalent:
		\begin{enumerate}
			\item $C_{\phi}(S_{\mu_1}H^p)\subseteq S_{\mu_2}H^p$
			\item $C_{\psi}((S_{\mu_1}\circ B_\omega )H^p)\subseteq (S_{\mu_2}\circ B_\omega )H^p$
			\item $\nu_2(B_\omega( E))\leq\nu_1(B_\omega(\phi(E)))$ for every Borel subset $E$ of $\mathbb{T}$.
			
		\end{enumerate}
	\end{theorem}
	\begin{proof}
		Using Theorem \ref{main}, we have
		\begin{equation*}
			\begin{split}
				C_{\phi}(S_{\mu_1}H^p)\subseteq S_{\mu_2}H^p&\Leftrightarrow  S_{\mu_1}\circ \phi=
S_{\mu_2}f \text { for some } f\in H^\infty\\
				&\Leftrightarrow  S_{\mu_1}\circ B_\omega  \circ \psi=(S_{\mu_2}f) \circ B_\omega \,\,
\big(=(S_{\mu_2}\circ B_\omega ) (f\circ B_\omega )\big)\\
				&\Leftrightarrow  C_{\psi}((S_{\mu_1}\circ B_\omega )H^p)\subseteq (S_{\mu_2}\circ B_\omega)H^p.\\
			\end{split}
		\end{equation*}
		This completes the proof of $(1) \Leftrightarrow(2)$.
		
		For $(2)\Leftrightarrow(3)$, using Theorem \ref{text}, we have
		\begin{equation*}
			\begin{split}
				C_{\psi}((S_{\mu_1}\circ B_\omega  )H^p)\subseteq (S_{\mu_2}\circ B_\omega  )H^p
&\Leftrightarrow C_{\psi}(S_{\nu_1}H^p)\subseteq S_{\nu_2}H^p\\
				&\Leftrightarrow   \nu_2(E)\leq\nu_1(\psi(E))\\
				&\Leftrightarrow \nu_2(B_\omega ( E))\leq\nu_1(\psi(B_\omega (E)))=\nu_1(B_\omega (\phi(E))),\\
			\end{split}
		\end{equation*}
		for all Borel subsets $E$ of $\mathbb{T}$, where $\nu_i(E)=\int_{\phi(E)}\frac{1-|\omega |^2}
{|\omega -t|^2}d\mu_i(t)$, for $i=1,2$.	
	\end{proof}

 Note that every function $f$ in $H^p$ can be factored (inner-outer factorization \cite[Theorem 2.8]{duren}) as $f=BSg$,
  where $B$ is a Blaschke product, $S$ is a singular inner function
 and $g$ is an outer function. It is important to note that these components are unique up
 to unimodular constants.

	\begin{theorem}
		Let $\phi$ be a self holomorphic map on $\mathbb{D}$ and let $S_{\mu_1}$ and $S_{\mu_2}$ be
 two singular inner functions. Then $C_{\phi}(S_{\mu_1}H^p)\subseteq S_{\mu_2}H^p$ if and only if
  $S_{\mu_2}$ divides singular part of $S_{\mu_1}\circ\phi$.
	\end{theorem}
	\begin{proof}
		Let $S_{\gamma}$ be the singular inner component of $S_{\mu_1}\circ\phi$. 
Now, let us assume
that
$C_{\phi}(S_{\mu_1}H^p)\subseteq S_{\mu_2}H^p$. Then by Theorem \ref{main},
$f=(S_{\mu_1}\circ\phi)/S_{\mu_2}\in H^{\infty}$. As $fS_{\mu_2}=S_{\mu_1}\circ\phi$,
comparing singular part on both sides of the equation and by the uniqueness of inner-outer
 factorization, we have $S_{\mu_2}$ divides $S_{\gamma}$.
 The converse part is trivial.
	\end{proof}

\section{Invariant Beurling subspaces of composition operators}
In this section, for given a Blaschke Beurling subspace we will try to find all
composition operators which makes it invariant. We have given two characterization for $C_\phi(BH^p)\subseteq BH^p$
in Section 3 (see Theorem \ref{main} and Proposition \ref{charb}). The first one is in
terms of $H^\infty$ functions and the next one is in terms of multiplicities of zeros of $B$.
 In Theorem \ref{main1}, we give a third characterization for $BH^p$ to be invariant under $C_\phi$ in
terms of  derivatives of $\phi$ at the zeros of $B$.
\begin{remark}\label{rem}
	Let $\theta$ be an inner function. If
	$C_{\phi}(\theta H^p)\subseteq\theta H^p$, then $\phi$ maps $Z(\theta)$ into itself
(see Section 3).
\end{remark}

It is natural to ask whether the converse of the above remark is true.
The answer is negative in general. 
 If the multiplicities of all the zeros of the Blaschke product are the same, then the answer is positive.
%
%
%
%
%

\begin{proposition}\label{necc}
	Let $B= \gamma\prod _{i\in\mathbb{N}}(\alpha_iB_{a_i})^{m}$ be an
	arbitrary Blaschke product with all of its zeros have the same multiplicity (say m).
	Then $C_{\phi}(BH^p)\subseteq B H^p$ if and only if  $\phi(\{a_i\}_{i\in\mathbb{N}})
\subseteq \{a_i\}_{i\in\mathbb{N}}$.
\end{proposition}
\begin{proof}
	The necessary part easily follows from the Remark \ref{rem}. For the sufficient part,
consider any $a_ j\in Z(B)$. If $\phi(a_j)=a_k$ for some $a_k\in Z(B)$ then as a consequence of Theorem \ref{RFT}, $\alpha_kB_{a_k} \circ \phi=B_{a_j}g$ for
 some $g\in H^{\infty}$.
	Thus, $$ B\circ \phi=(\alpha_kB_{a_k}\circ \phi)^m h=B_{a_j}^mg^mh,$$ where
$h=\gamma\prod_{i\in\mathbb{N}, i\neq k}(\alpha_iB_{a_i}\circ \phi)^m$.
	Therefore, $mult_{B\circ\phi}(a_j)\geq m=mult_{B}(a_j)$. Hence, by Proposition
\ref{charb}, $C_{\phi}(BH^p)\subseteq B H^p$.
\end{proof}

The assumption  on multiplicities in Proposition \ref{necc} is essential.
 To highlight this, we have the following example.

\begin{example}
	Consider the Blaschke product  $B=\gamma \prod_{i\in\mathbb{N}}(\alpha_iB_{a_i})^{m_i}$.
Suppose $m_k< m_j$  for some  $k, j\in \mathbb{N}$. Take $\phi=B_{a_k}\circ B'$,
 where $B'=\prod_{i\in\mathbb{N}}\alpha_iB_{a_i}$.
 Then
	$\phi(a_i)=a_k$ for all $i\in \mathbb{N}$ and
	\begin{equation*}
		\begin{split}
			B\circ \phi&=B\circ B_{a_k}\circ B'
			=\gamma \prod_{i\in\mathbb{N}}(\alpha_iB_{a_i}\circ B_{a_k}\circ B' )^{m_i}\\
			&=\gamma(\alpha_kB_{a_k}\circ B_{a_k}\circ B' )^{m_k}\prod_{i\in\mathbb{N}, i\neq k}
(\alpha_i B_{a_i}\circ B_{a_k}\circ B' )^{m_i}\\
			&=\gamma\alpha_k^{m_k}\prod_{i\in\mathbb{N}}(\alpha_iB_{a_i})^{m_k}
\prod_{i\in\mathbb{N}, i\neq k}(\alpha_iB_{a_i}\circ B_{a_k}\circ B' )^{m_i}.\\	
		\end{split}
	\end{equation*}
	The last equality follows from the fact that
$B_{a_k}\circ B_{a_k}$ is identity.
	Since the second product $\prod_{i\in\mathbb{N}, i\neq k}
(\alpha_iB_{a_i}\circ B_{a_k}\circ B' )^{m_i}$ is nonzero at $a_j$, it does not have the factor
 $B_{a_j}$. Thus, the multiplicity of ${a_j}$ in $B\circ \phi$ is $m_k$,
  which is strictly less than $m_j$.
  Therefore, $C_{\phi}(BH^p)\nsubseteq BH^p$ by Proposition \ref{charb}.
  Hence, if the multiplicities of any two zeros are different then there exists
  $\phi\in \mathcal{S}$ such that $\phi(\{a_i\}_{i\in \mathbb{N}})
  \subseteq\{a_i\}_{i\in \mathbb{N}}$  and
    $BH^p$ is not invariant under $C_{\phi}$.
\end{example}

	For any inner function $\theta$,  the collection of all $\phi\in \mathcal{S}$
such that $C_{\phi}(\theta H^p)\subseteq\theta H^p$ is denoted by
	$\mathcal{L}_{\theta}$.
For any inner function $\theta $, $\mathcal{L}_{\theta}$ is always nonempty since the composition
under the identity map on $\mathbb{D}$ makes the subspace $\theta H^p$ invariant. We will refer to
the identity map as the trivial map. We are interested in the existence of a nontrivial element in
$\mathcal{L}_{\theta}$. For any Blaschke product, the following result will give us a partial answer.
\begin{theorem}
	Let $B=\gamma \prod_{i\in\mathbb{N}}(\alpha_iB_{a_i})^{m_i}$ be a Blaschke product  with
$\max_{i\in\mathbb{N}}\{m_i\}$ is finite. Then there exists a nontrivial map $\phi$
so that $BH^p$ is invariant under $C_{\phi}$.
\end{theorem}
\begin{proof}
	Let $m_k=\max_{i\in\mathbb{N}}\{m_i\}$. Define $\phi=B_{a_k}\circ B$. Then
	\begin{equation*}
		\begin{split}
			B\circ \phi&=B\circ B_{a_k}\circ B
			=\gamma\prod_{i\in\mathbb{N}}(\alpha_iB_{a_i}\circ B_{a_k}\circ B )^{m_i}\\
			&=\gamma(\alpha_kB_{a_k}\circ B_{a_k}\circ B )^{m_k}
\prod_{i\in\mathbb{N}, i\neq k}
(\alpha_iB_{a_i}\circ B_{a_k}\circ B )^{m_i}\\
			&=\gamma (\alpha_kB)^{m_k}\prod_{i\in\mathbb{N}, i\neq k}(\alpha_iB_{a_i}\circ
B_{a_k}\circ B )^{m_i}.\\	
		\end{split}
	\end{equation*}
	Thus, we can easily see that ${mult}_B(a_i)\leq mult_{B\circ \phi}(a_i)$ for all
$i\in \mathbb{N}$. Hence by Proposition \ref{charb}, we see that $\phi\in \mathcal{L}_B$.
\end{proof}
\begin{corollary}
	If $B$ is a finite Blaschke product then $\mathcal{L}_B$ is nontrivial.
\end{corollary}

\begin{theorem}\label{auto}
	Let $B=\gamma \prod_{i\in\mathbb{N}}(\alpha_iB_{a_i})^{m_i}$ be Blaschke product and
let $\phi\in \mathcal{L}_B$ be a disk automorphism. Then $mult_B(a_j)\leq mult_B(\phi (a_j))$
 for all $j\in \mathbb{N}$.
\end{theorem}
\begin{proof}
	Let $\phi\in \mathcal{L}_B$ be a disk automorphism. Fix  $j\in \mathbb{N}$ and suppose
$\phi(a_j)=a_k$  for some
	$k\in \mathbb{N}$. Then, $(\alpha_kB_{a_k}\circ\phi)(a_j)=0$. Since both $\phi$ and
 $\alpha_kB_{a_k}$ are disk automorphisms, we have $\alpha_kB_{a_k}\circ \phi=\lambda B_{a_j}$
 for some $\lambda\in \mathbb{T}$. Now,
	\begin{equation*}
		\begin{split}
			B\circ \phi=\gamma\prod_{i\in\mathbb{N}}(\alpha_i B_{a_i}\circ \phi)^{m_i}
			&=\gamma( \alpha_kB_{a_k}\circ \phi )^{m_k}\prod_{i\in\mathbb{N}, i\neq k}(\alpha_iB_{a_i}\circ \phi )^{m_i}\\
			&=\gamma' B_{a_j}^{m_k}\prod_{i\in\mathbb{N}, i\neq k}(\alpha_iB_{a_i}\circ \phi )^{m_i},\\
		\end{split}
	\end{equation*}
	where $\gamma'=\gamma \lambda^{m_k}$.
	Since $(\alpha_iB_{a_i}\circ \phi)(a_j)\neq 0$ for any $i\neq k$, $mult_{B\circ \phi}(a_j)=m_k$,
 and by Proposition \ref{charb}, we get $m_j\leq m_k$.
\end{proof}

By the above theorem, it is easy to observe that if any $\phi \in  \mathcal{L}_B$ maps some zero
of $B$ with greater multiplicity to some other zero of $B$ with lower multiplicity then $\phi$
cannot be a disk automorphism.

\begin{corollary}\label{corrr}
	Let $B=\gamma\prod_{i=1}^{n}B_{a_i}^{m_i}$, $n>1$ and without loss of generality,
let $m_1\leq m_2 \leq \cdots \leq m_n$. If $m_{n-2}<m_{n-1}<m_n$ ($m_1<m_2$ in case $n=2$), then  $\mathcal{L}_B$ 
 does not contain any nontrivial automorphism.
 \end{corollary}

\begin{proof}
	
	Let $\phi\in \mathcal{L}_B$ be a disk automorphism. By Theorem \ref{auto}, $\phi(a_n)=a_n$ and
$\phi(a_{n-1})=a_{n-1}$. Therefore $\phi$ has to be the identity map on $ \mathbb{D}$.
\end{proof}	




As a consequence, we get the following result.
\begin{theorem}
	Let $\theta$ be any inner function such that the Blaschke component of $\theta$ satisfies the
hypothesis of Corollary \ref{corrr}. Then
	$\theta H^p$ is not invariant under $C_{\phi}$ for any nontrivial disk automorphism $\phi$.
\end{theorem}

\begin{proof}
	Let $B$ be the Blaschke component of $\theta$ and let $B$ satisfies the hypothesis of
 Corollary \ref{corrr}. Suppose $\phi \in  \mathcal{L}_\theta$. Remark \ref{zero} tell us that $\phi \in  \mathcal{L}_B$. Then by Corollary \ref{corrr}, $\phi$ cannot be a nontrivial automorphism.
\end{proof}

Now we will give a characterization for $BH^p$ to be invariant under a composition
operator $C_\phi$ in terms of the value of derivatives of $\phi$ at the zeros of $B$. Before that,
 we will prove a lemma which will act as an important tool for proving the mentioned result.
 For a function $f$ and $n\in \mathbb{N}$, we denote  the $n^{\text{th}}$ derivative of $f$ at $z$ by $f^{(n)}(z)$. Also, we use $f^{(0)}(z)$ to denote $f(z)$.

\begin{lemma}\label{phi}
	Let $\phi\in \mathcal{S}$ such that $\phi(a)=b$ for $a,b\in\mathbb{D}$. For any $k\in \mathbb{N}$,
	$$(B_b\circ\phi)^{(l)}(a)=0   \, \text{ for all }   \, \, 1\leq l\leq k \text{ if and only if }
	\phi^{(l)}(a)=0  \, \, \text{for all }  \, \, 1\leq l\leq k.$$
\end{lemma}
\begin{proof}
	Let $\phi\in \mathcal{S}$ such that $\phi(a)=b$. Thus, we get $(B_b\circ \phi)(a)=0$. 
Note that
$(B_b\circ\phi)^{(1)}(a)=B_b^{(1)}(\phi(a))\phi^{(1)}(a)$.  For any $ q\in \mathbb{N}$, by
Leibniz rule for differentiation of product of two functions, we have
	\begin{equation*}
		\begin{split}			(B_b\circ\phi)^{(q)}(a)&=\Big(\big(B_b^{(1)}\circ\phi\big)\phi^{(1)}\Big)^{(q-1)}(a)
			=\sum_{ r=0}^ {q-1}{{q-1}\choose{r}}\Big((B_b^{(1)}\circ\phi)^{(r)}\phi^{(q-r)}\Big)(a)\\
			&=B_b^{(1)}(\phi(a))\phi^{(q)}(a)+\sum_{ r=1}^ {q-1}{{q-1}\choose{r}}\Big((B_b^{(1)}
\circ\phi)^{(r)}\phi^{(q-r)}\Big)(a).\\
		\end{split}
	\end{equation*}	
	Since $B_b^{(1)}$ is a non-vanishing function, by using the above equation and principle
of mathematical induction, the desired result follows. 
\end{proof}
\begin{remark}
	In view of Lemma \ref{phi}, Theorem \ref{auto} still holds under the
following weaker assumption:
$\phi \in  \mathcal{L}_B$  and  $\phi^{(1)}(a_j)\neq 0$ for all $j\in \mathbb{N}$.
\end{remark}

For any $x\in \mathbb{R}$, we denote $\lceil x\rceil$ for the lowest integer greater than or equal to x.

\begin{theorem}\label{deriv}
	Let $B=\gamma \prod_{i\in\mathbb{N}}(\alpha_iB_{a_i})^{m_i}$ be a Blaschke product and $\phi \in  \mathcal{S}$. Fix
 $j\in \mathbb{N}$ and suppose $\phi(a_j)=a_k$ for some $k\in \mathbb{N}$. Then
 $ {mult}_{B}(a_j)\leq{mult}_{B\circ \phi}(a_j)$ if and only if  $(B_{a_k}\circ\phi)^{(l)}(a_j)=0 $
 for $0\leq l\leq {\Big\lceil\frac{m_j}{m_k}\Big\rceil}-1.$
	Moreover, if ${\Big\lceil\frac{m_j}{m_k}\Big\rceil}>1,$
	then these conditions are equivalent to
	$\phi^{(l)}(a_j)=0 $ for all $ 1\leq l\leq {\Big\lceil\frac{m_j}{m_k}\Big\rceil}-1. $
\end{theorem}
\begin{proof}
	Fix $j\in \mathbb{N}$. Suppose $\phi(a_j)=a_k$ for some $k\in \mathbb{N}$. Then
$\big(B_{a_k}\circ \phi\big)(a_j)=0$ and $\big(B_{a_i}\circ \phi\big)(a_j)\neq 0$ for any
 $i \neq k$. Also, we have 	$$B\circ \phi=
	\gamma\prod_{i\in\mathbb{N}}(\alpha_iB_{a_i}\circ \phi)^{m_i}=
	\gamma(\alpha_kB_{a_k}\circ \phi )^{m_k}\prod_{i\in\mathbb{N}, i\neq k }
(\alpha_iB_{a_i}\circ \phi )^{m_i}.$$
This yields
	 that ${mult}_{B\circ \phi}(a_j)={{mult}}_{(B_{a_k}\circ \phi)^{m_k}}(a_j)$.
 It is easy to verify that ${mult}_{(B_{a_k}\circ \phi)^{m_k}}(a_j)=m_k {mult}_{B_{a_k}\circ \phi}(a_j)$.
	Thus,
	\begin{equation*}
		\begin{split}	{mult}_{B\circ \phi}(a_j)\geq {mult}_{B}(a_j)
			& \Leftrightarrow m_k{mult}_{B_{a_k}\circ \phi}(a_j)\geq m_j
			\Leftrightarrow {mult}_{B_{a_k}\circ \phi}(a_j)\geq \Big\lceil\frac{m_j}{m_k}\Big\rceil\\
			&\Leftrightarrow (B_{a_k}\circ\phi)^{(l)}(a_j)=0 \text{ for }0\leq l\leq
 {\Big\lceil\frac{m_j}{m_k}\Big\rceil}-1.\\
		\end{split}
	\end{equation*}
	Furthermore, if ${\Big\lceil\frac{m_j}{m_k}\Big\rceil}>1$, then by  Lemma \ref{phi}
 the above statement is equivalent to 	
	$\phi^{(l)}(a_j)=0 $ for all $1\leq l\leq {\Big\lceil\frac{m_j}{m_k}\Big\rceil}-1 $.
\end{proof}

Using Proposition \ref{charb} and Theorem \ref{deriv}, we have the following result.

\begin{theorem}\label{main1}
	Let $B=\gamma \prod_{i\in\mathbb{N}}(\alpha_iB_{a_i})^{m_i}$ be a Blaschke product and
$\phi\in \mathcal{S}$ with $\phi(\{a_i\}_{i\in\mathbb{N}})\subseteq\{a_i\}_{i\in\mathbb{N}}.$
Then  $C_{\phi}(BH^p)\subseteq BH^p$ if and only if $\phi^{(l)}(a_i)=0 $ for all $1\leq l\leq{\Big\lceil\frac{{mult}_B(a_i)}{mult_{B}(\phi{(a_i)})}\Big\rceil}-1$, whenever
${mult}_B(a_i)>{mult}_B(\phi(a_i)). $

\end{theorem}

In particular, if the Blaschke product has exactly two zeros then the following 
result holds.
\begin{corollary}\label{inva}
	Let $B=\gamma B_{a_1}^mB_{a_2}^n$ with $m> n$ and let $\phi\in \mathcal{S}$
with
$\phi(a_1)=a_2$.
	Then  $C_{\phi}(BH^p)\subseteq BH^p$ if and only if
	$ \phi^{(l)}(a_1)=0 $ for all $1\leq l\leq \Big\lceil\frac{m}{n}\Big\rceil-1. $
\end{corollary}

Now, we will focus solely on the finite Blaschke products. We will be
completely describing the set $\mathcal{L}_B$ for the Blaschke products
 with exactly one zero as well as two zeros.

\begin{theorem}
	Let $B=\gamma B_a^m$, where $a\in \mathbb{D}$ and $m\in \mathbb{N}$. Then
	$$ \mathcal{L}_B=\{B_a \circ\psi\circ B_a: \psi(0)=0 \text{ and } \psi\in \mathcal{S}\}.  $$
\end{theorem}
\begin{proof}
	By Proposition \ref{necc}, $\phi\in \mathcal{L}_B$ if and only if $\phi(a)=a$. It can be
easily seen that $\phi(a)=a$ if and  only if $\phi=B_a \circ\psi\circ B_a$ for some
$\psi\in \mathcal{S}$ with $\psi(0)=0$.
\end{proof}


\begin{theorem}\label{fix}
	Let $B=\gamma \prod_{i=1}^{n}B_{a_i}^{m_i}$ and let $\phi\in \mathcal{S}$ such that for all
 $i$, $\phi(a_i)=a_j$ for some $j$. Then $C_{\phi}(BH^p)\subseteq BH^p$ if and
  only if $\phi=B_{a_j} \circ \big(h\prod_{i=1}^{n} B_{a_i}^{k_i}\big)$, for some
  $h\in \mathcal{B}_1$ and $k_i=\lceil \frac{m_i}{m_j}\rceil$. 	
\end{theorem}
\begin{proof}
Suppose $\phi\in \mathcal{S}$ such that for all $i$, $\phi(a_i)=a_j$ for some $j$.
 It implies that $(B_{a_j} \circ \phi)(a_i)=0$ for  $1\leq i\leq n$.
 Thus,
$B_{a_j} \circ \phi=g\prod_{i=1}^{n} B_{a_i}$ for some $g\in \mathcal{B}_1$. Therefore,
%
%
%
%
$\phi=B_{a_j}\circ f$, where $f=g\prod_{i=1}^{n} B_{a_i}$ for
 some $g\in \mathcal{B}_1$. 	
 
	Fix an $a_i$. As in the proof of Lemma \ref{phi}, we can verify that, for any $k\in \mathbb{N}$,
 $\phi^{(l)}(a_i)=0$ { for all } $1\leq l\leq k$ { if and only if } $f^{(l)}(a_i)=0$ { for all } $1\leq l\leq k.$

	Next we claim that for any $k\in \mathbb{N}$, $f^{(l)}(a_i)=0$ { for all } $1\leq l\leq k$
 { if and only if } $g^{(l)}(a_i)=0$ { for all } $0 \leq l\leq k-1.$ Now for any $ q\in \mathbb{N}$,
  by generalizing the Leibniz rule for differentiation of product of functions
  and by using $B_{a_i}(a_i)=0$, we have
	
	\begin{equation*}
		\begin{split}
			f^{(q)}(a_i)&=\sum_ {q_0+q_1+q_2+\cdot\cdot\cdot +q_n =q}{{q}\choose {q_0,q_1,q_2,\ldots,q_n}}\big(g^{(q_0)}B_{a_1}^{(q_1)}B_{a_2}^{(q_2)}\cdot\cdot\cdot B_{a_n}^{(q_n)}\big)(a_i)\\
			&=\sum_{\substack {q_0+q_1+q_2+\cdot\cdot\cdot+q_n =q\\q_i>0}}{{q}\choose {q_0,q_1,q_2,\ldots,q_n}}\big(g^{(q_0)}B_{a_1}^{(q_1)}B_{a_2}^{(q_2)}\cdot\cdot\cdot B_{a_n}^{(q_n)}\big)(a_i)\\
			&={{q}\choose {q-1,0,\ldots,1,\ldots,0,0}}\Big(g^{(q-1)}B_{a_i}^{(1)}\prod_{{r=1},r\neq i}^{n}B_{a_r}\Big)(a_i)\\
&\hspace{0.5cm} +			\sum_{\substack {q_0+q_1+q_2+\cdot\cdot\cdot+q_n=q\\q_i\geq1,q_0<q-1}}{{q}\choose {q_0,q_1,q_2,\ldots,q_n}}\big(g^{(q_0)}B_{a_1}^{(q_1)}B_{a_2}^{(q_2)}\cdot\cdot\cdot B_{a_n}^{(q_n)}\big)(a_i),\\
		\end{split}
	\end{equation*}
	where ${{q}\choose {q_0,q_1,q_2,\ldots,q_n}}=\frac{q!}{q_0!q_1!q_2!\cdot\cdot\cdot q_n!}$.
	As similar to the proof in Lemma \ref{phi}, using the above equation we can prove our claim.

Therefore, for any $a_i$, we have  $\phi^{(l)}(a_i)=0 $ for all $ 1\leq l\leq k$ if and only if
$g^{(l)}(a_i)=0,$  for all $0 \leq l\leq k-1$,  which is equivalent to the statement $g=B_{a_i}^kg_1$,
for some $g_1\in \mathcal{B}_1.$
Thus by  Theorem \ref{main1}, we can conclude that
$ C_{\phi}(BH^p)\subseteq BH^p$ if and only if $\phi=B_{a_j} \circ \big(h\prod_{i=1}^{n} B_{a_i}^{k_i}\big)$
 for some $h\in \mathcal{B}_1$ and $k_i=\lceil \frac{m_i}{m_j}\rceil$.
\end{proof}

\begin{proposition}\label{autom}
Given any finite points $\{a_1,a_2,\ldots,a_n\}\subseteq \mathbb{D}$, there exists at most one self
 map $\phi$ of $\mathbb{D}$ such that
$\phi(a_i)=a_{i+1}$ for  $1\leq i <n$ and $\phi(a_n)=a_1$. If such a map exists, then it will be an
 elliptic automorphism.
\end{proposition}
\begin{proof}

Let  $\phi\in \mathcal{S}$ be such that $\phi(a_i)=a_{i+1}$ for all $1\leq i <n$ and $\phi(a_n)=a_1$.
Then $\phi^{[n]}(a_i)=a_i$, for all $1\leq i \leq n$ (Here $\phi^{[n]}$ denotes n times composition of
 the function $\phi$). Since $\phi^{[n]}$ have more than one fixed point, we have $\phi^{[n]}=I$, the
  identity map. That is, $\phi^{[n-1]}\circ \phi=\phi\circ\phi^{[n-1]}=I$ which implies the $\phi$ is
  an automorphism. 
  
  If such a map exists, the uniqueness can be verified easily.
  Also by Denjoy-Wolff theorem \cite[Section 5.1]{shapiro1}, $\phi$ has a unique
   fixed point inside $\mathbb{D}$. Hence $\phi$ is an elliptic automorphism.
\end{proof}
\begin{theorem}
Let $\theta $ be an inner function having only a finite number of zeros on $\mathbb{D}$. Then every
 nontrivial disk automorphism $\phi$ such that  $C_{\phi}(\theta H^p)\subseteq \theta H^p$ will be
 an elliptic automorphism.
\end{theorem}
\begin{proof}
Let $\theta $ be an inner function with zeros $ a_1,a_2,\ldots,a_n$ and let $\phi$ be a nontrivial
disk automorphism such that $C_{\phi}(\theta H^p)\subseteq \theta H^p$. By the Remark \ref{rem}, we see that
 $\phi$ is a bijection on the zero set $\{ a_1,a_2,\ldots,a_n\}$. Thus, there exists a subset
 $\{b_1,b_2,\ldots,b_k\}$ of the zeros with $k>1$, such that $\phi(b_i)=b_{i+1}$ for all
 $1\leq i <k$ and $\phi(b_k)=b_1$. By Proposition \ref{autom},  $\phi$ will be an elliptic automorphism.
\end{proof}
For $a,b\in \mathbb{D}$, consider the map $\phi_{a,b}=B_{a}\circ B_c\circ B_{a}$, where $c=B_{a}({b})$.
We can easily verify that the above map is an automorphism on $\mathbb{D}$ which interchanges
 $a$ and $b$. Hence the  Proposition \ref{autom} will help us to conclude the following
 result.
\begin{corollary}\label{inter}
Let $a,b\in \mathbb{D}$ and $\phi\in \mathcal{S}
$ such that $\phi(a)=b$ and $\phi(b)=a$. Then $\phi=\phi_{a,b}$.
\end{corollary}
Now we will focus on the Blaschke product of the form $B=\gamma B_a^mB_{b}^n$. Without loss of
generality let $m\geq n$.
By Remark \ref{rem}, any $\phi\in \mathcal{L}_B$  maps $\{a,b\}$ into $\{a,b\}$. We will discuss
all possibilities case by case.

Case 1 ($\phi(a)=a$ and $\phi(b)=b$):  Since $\phi$ has  more than one fixed point, $\phi$ will
be the identity map.

Case 2 ($\phi(a)=b$ and $\phi(b)=a$):
By Corollary \ref{inter}, $\phi=\phi_{a,b}$. For $\phi\in \mathcal{L}_B$, we should have $m=n$
 by Theorem \ref{auto}.

Case 3 ($\phi(a)=\phi(b)=b$): By Theorem \ref{fix}, we have
$ \phi\in \mathcal{L}_B$ if and only if $\phi=B_{b}\circ (B_{a}^kB_{b}g)$ for some $g\in \mathcal{B}_1,$
where $k=\lceil\frac{m}{n}\rceil$.\\
Case 4 ($\phi(a)=\phi(b)=a$): By Theorem \ref{fix}, we have
$\phi\in \mathcal{L}_B$ if and only if $\phi=B_{a}\circ (B_{a}B_{b}g)$ for some $ g\in \mathcal{B}_1.$

Based on the above discussion, we have the following conclusions. Here $I$ denotes
the identity map.
\begin{theorem}
Let $B=\gamma(B_{a}B_{b})^n$ for some $n\in \mathbb{N}$. Then
$$\mathcal{L}_B= \cup_{g\in\mathcal{B}_1}\{B_{a}\circ( B_{a}B_{b}g),B_{b}\circ (B_{a}B_{b}g)\}\cup\{I,\phi_{a,b}\}.$$
\end{theorem}
\begin{theorem}
Let $B=\gamma B_{a}^mB_{b}^n$ with $m> n$. Then
$$\mathcal{L}_B= \cup_{g\in \mathcal{B}_1}\{B_{a}\circ (B_{a}B_{b}g),B_{b}\circ (B_{a}^kB_{b}g)\} \cup \{I\},$$
where $k=\lceil\frac{m}{n}\rceil$.
\end{theorem}

{\bf Acknowledgments.} The first author is supported by the University-JRF Scheme by Cochin University
 of Science and Technology. The third author is supported by the Teachers Association for Research
  Excellence (TAR/2022/000063) of SERB (Science and Engineering Research Board), India.
\nocite{*} 
\bibliographystyle{amsplain}
\bibliography{database}
\end{document}